\documentclass{amsart}
\usepackage{amsmath, amssymb, amsthm}

\newtheorem{thm}{Theorem}
\newtheorem{lem}[thm]{Lemma}
\newtheorem{pro}[thm]{Proposition}

\def\Diff{{\rm Diff}}
\def\Ham{{\rm Ham}}
\def\id{{\rm id}}
\def\O{{\mathcal O}}
\def\R{{\mathbb R}}
\def\Symp{{\rm Symp}}
\def\Z{{\mathbb Z}}

\title[$C^0$-rigidity of contact diffeomorphisms]{Gromov's alternative, contact shape, and \\ $C^0$-rigidity of contact diffeomorphisms}
\author[S.~M\"uller \& P.~Spaeth]{Stefan M\"uller and Peter Spaeth}
\address{University of Illinois at Urbana-Champaign, Urbana, IL 61801}
\email{stefanm@illinois.edu}
\address{Penn State University, Altoona, PA 16601}
\email{spaeth@psu.edu}

\begin{document}
\thispagestyle{plain}

\begin{abstract}
We prove that the group of contact diffeomorphisms is closed in the group of all diffeomorphisms in the $C^0$-topology.
By Gromov's alternative, it suffices to exhibit a diffeomorphism that can not be approximated uniformly by contact diffeomorphisms.
Our construction uses Eliashberg's contact shape.
\end{abstract}

\maketitle

\section{Introduction and main result} \label{sec:intro}
Let $(W^{2 n},\omega)$ be a connected symplectic manifold, and $\omega^n$ be the corresponding canonical volume form.
We assume throughout this paper that all diffeomorphisms and vector fields are compactly supported in the interior of a given manifold.

A symplectic diffeomorphism obviously preserves the volume form $\omega^n$, and thus so does a diffeomorphism that is the uniform limit of symplectic diffeomorphisms.
In the early 1970s, Gromov discovered a fundamental hard versus soft alternative \cite{gromov:pdr86}:
either the group $\Symp (W,\omega)$ of symplectic diffeomorphisms is $C^0$-closed in the group $\Diff (W)$ of all diffeomorphisms (hardness or rigidity), or its $C^0$-closure is (a subgroup of finite codimension in) the group $\Diff (W,\omega^n)$ of volume preserving diffeomorphisms (softness or flexibility).
See section~\ref{sec:gromov-alternative} for a more precise statement.

Gromov's alternative is a fundamental question that concerns the very existence of symplectic topology \cite{mcduff:ist98, donaldson:gld90}.
That rigidity holds was proved by Eliashberg in the late 1970s \cite{eliashberg:rsc82, eliashberg:tsv87}.
One of the most geometric expressions of this $C^0$-rigidity is Gromov's non-squeezing theorem.
Denote by $B^{2 n} (r)$ a closed ball of radius $r$ in $\R^{2 n}$ with its standard symplectic form $\omega_0$, and by $B^2 (R) \times \R^{2 n -2} = Z^{2 n} (R) \subset \R^{2 n}$ a cylinder of radius $R$ so that the splitting $\R^2 \times \R^{2 n - 2}$ is symplectic.

\begin{thm}[Gromov's non-squeezing theorem \cite{gromov:pcs85}] \label{thm:non-squeezing}
If there exists a symplectic embedding of $B^{2 n} (r)$ into $Z^{2 n} (R)$, then $r \le R$.
\end{thm}

It follows at once that $\Symp (\R^{2 n},\omega_0)$ is $C^0$-closed in $\Diff (\R^{2 n})$.
Using Darboux's theorem, one can easily deduce that rigidity holds for general symplectic manifolds.

\begin{thm}[$C^0$-rigidity of symplectic diffeomorphisms \cite{eliashberg:rsc82, eliashberg:tsv87}] \label{thm:symp-rigidity}
The group $\Symp (W,\omega)$ is $C^0$-closed in the group $\Diff (W)$ of all diffeomorphisms.
\end{thm}

See section~\ref{sec:gromov-alternative} for details.
It was observed later that symplectic capacities give rise to a $C^0$-characterization of symplectic diffeomorphisms, and in particular to proofs of Theorem~\ref{thm:symp-rigidity} that do not rely on Gromov's alternative \cite{mcduff:ist98, hofer:sih94}.
Note that the existence of symplectic capacities is equivalent to Theorem~\ref{thm:non-squeezing}.

Let $(M,\xi)$ be a connected contact manifold of dimension $2 n - 1$, that is, $\xi \subset TM$ is a completely non-integrable codimension one tangent distribution.
In this paper, all contact structures are assumed to be coorientable, i.e.\ there exists a $1$-form $\alpha$ with $\ker \alpha = \xi$ and $\alpha \wedge (d\alpha)^{n - 1}$ is a volume form on $M$.
Moreover, we always fix a coorientation of $\xi$, so that the contact form $\alpha$ is determined up to multiplication by a positive function, and unless explicit mention is made to the contrary, a contact diffeomorphism is assumed to preserve the coorientation of $\xi$.

In the contact case, the analog of Gromov's alternative is the following:
the group $\Diff (M,\xi)$ of contact diffeomorphisms is either $C^0$-closed or $C^0$-dense in $\Diff (M)$.
See section~\ref{sec:gromov-alternative} for a precise statement.
A popular folklore theorem among symplectic and contact topologists is that an analog of Theorem~\ref{thm:symp-rigidity} holds in the contact case.
However, up to this point no proof has been published.
The purpose of this paper is to remedy this situation by proving the following main theorem.

\begin{thm}[$C^0$-rigidity of contact diffeomorphisms] \label{thm:contact-rigidity}
The contact diffeomorphism group $\Diff (M,\xi)$ is $C^0$-closed in the group $\Diff (M)$ of all diffeomorphisms.
That is, if $\varphi_k$ is a sequence of contact diffeomorphisms of $(M,\xi)$ that converges uniformly to a diffeomorphism $\varphi$, then $\varphi$ preserves the contact structure $\xi$.
\end{thm}

Similar to the case of symplectic diffeomorphisms, to prove Theorem~\ref{thm:contact-rigidity} we single out an invariant of contact diffeomorphisms that is (at least in the special case that is considered below) continuous with respect to the $C^0$-topology.
In section~\ref{sec:rigidity} we will use the contact shape defined in \cite{eliashberg:nio91} to construct a diffeomorphism that cannot be approximated uniformly by contact diffeomorphisms.
The construction is local and therefore applies to any contact manifold.
We recall Eliashberg's shape invariant in section~\ref{sec:contact-shape}.
We suspect that there are alternate proofs of Theorem~\ref{thm:contact-rigidity} based on other contact invariants.
It would also be interesting to discover a $C^0$-characterization of contact diffeomorphisms (cf.\ \cite{mcduff:ist98}) similar to that of (anti-) symplectic diffeomorphisms, and then to find a proof of contact rigidity that does not use Gromov's alternative.

To make some final remarks, let $\varphi_k$ and $\varphi$ be as in the rigidity theorem.
Then $\varphi_k^* \alpha = e^{g_k} \alpha$ and $\varphi^* \alpha = e^g \alpha$ for smooth functions $g_k$ and $g$ on $M$.
The hypothesis of Theorem~\ref{thm:contact-rigidity} does not guarantee that the sequence $g_k$ converges to $g$ even pointwise, see \cite{ms:tcd1, ms:tcd2}.
The following theorem is a special case of a contact rigidity theorem that is proved in the context of topological contact dynamics in \cite{ms:tcd1}.

\begin{thm}[{$C^0$-rigidity of contact diffeomorphisms and conformal factors \cite{ms:tcd1}}] \label{thm:contact-cf-rigidity}
Let $(M,\xi = \ker \alpha)$ be a contact manifold with a contact form $\alpha$, and $\varphi_k$ be a sequence of contact diffeomorphisms of $(M,\xi)$ that converges uniformly to a diffeomorphism $\varphi$, such that $\varphi_k^* \alpha = e^{g_k} \alpha$, and $g_k$ converges to a (a priori not necessarily smooth) function $g$ uniformly.
Then $\varphi$ is a contact diffeomorphism, the function $g$ is smooth, and $\varphi^* \alpha = e^g \alpha$.
In particular, if $\{ \varphi_H^t \}$ and $\{ \varphi_F^t \}$ are two contact isotopies that are generated by smooth functions $H$ and $F$ on $M$, respectively, then $\varphi_H^t = \varphi \circ \varphi_F^t \circ \varphi^{-1}$ for all times $t$ if and only if $H_t = e^{- g} (F_t \circ \varphi)$ for all $t$.
\end{thm}

This result is of particular relevance in contact dynamics, and should be viewed as a complement to Theorem~\ref{thm:contact-rigidity}.
It is worthwhile to note that neither Theorem~\ref{thm:contact-rigidity} nor Theorem~\ref{thm:contact-cf-rigidity} implies the other result.

\section*{Acknowledgements}
We would like to thank Yasha Eliashberg for many helpful discussions regarding the proof of $C^0$-rigidity of contact diffeomorphisms, in particular during a recent conference at ETH Z\"urich in June 2013.

\section{Gromov's alternative} \label{sec:gromov-alternative}
Recall that the group $\Ham (W,\omega)$ of Hamiltonian diffeomorphisms is comprised of those diffeomorphisms that are time-one maps of (time-dependent) vector fields $X$ such that the $1$-forms $\iota_{X_t} \omega = \omega (X_t, \cdot)$ are exact for all times $t$.

\begin{thm}[{Gromov's maximality theorem \cite[Section~3.4.4~H]{gromov:pdr86}}]
Let $(W^{2 n},\omega)$ be a connected symplectic manifold.
Suppose that a (not necessarily closed) subgroup $G \subset \Diff (W,\omega^n)$ contains all Hamiltonian diffeomorphisms, and that $G$ contains an element $\psi$ that is neither symplectic nor anti-symplectic, i.e.\ $\psi^* \omega \not= \pm \omega$.
Then $G$ contains the subgroup of volume preserving diffeomorphisms with zero flux.
\end{thm}

We do not recall the flux map in detail, but note that it is a homomorphism from the connected component of the identity of $\Diff (W,\omega^n)$ into a quotient of the group $H^{2 n - 1} (W,\R)$ by a discrete subgroup.
Its kernel is comprised of those diffeomorphisms that are generated by (time-dependent) exact vector fields $X$, i.e.\ such that the $(2 n - 1)$-form $\iota_{X_t} \omega^n$ is exact for all times $t$.
In particular, the time-one map of a divergence-free vector field that is supported in a coordinate chart has vanishing flux.
The flux of a Hamiltonian diffeomorphism is also equal to zero.
See \cite{banyaga:scd97} for details.

The easier part of the proof uses linear algebra to prove that there are enough diffeomorphisms $\psi_1, \ldots, \psi_N \in G$ so that every exact vector field $X$ can be decomposed uniquely into a sum $X = X_1 + \ldots + X_N$, where the vector field $X_k$ is Hamiltonian with respect to the symplectic form $\psi_k^* \omega$.
Then the linearization of the map $D \colon \Symp (W,\omega) \times \ldots \times \Symp (W,\omega) \to \Diff (W,\omega^n)$, $(\varphi_1, \ldots, \varphi_N) \mapsto \psi_1 \circ \varphi_1 \circ \ldots \circ \psi_N \circ \varphi_N$, is surjective near the identity, and also has a right inverse.
The same techniques used to prove Gromov's difficult implicit function theorem then show that the image of $D$ contains an open neighborhood of the identity in the subgroup of all diffeomorphisms that are generated by (time-dependent) exact vector fields.
But this open neighborhood generates the entire subgroup of diffeomorphisms with zero flux, hence the proof.
The assumption that all vector fields are compactly supported in the interior guarantees that their flows are well-defined and exist for all times.

In order to apply Gromov's maximality theorem in the proof of Theorem~\ref{thm:symp-rigidity}, we need the following lemma.

\begin{lem} \label{lem:no-anti-symp}
Let $(W^{2 n},\omega)$ be a connected symplectic manifold.
Then the $C^0$-closure of the group $\Symp (W,\omega)$ of symplectic diffeomorphisms in the group $\Diff (W)$ of all diffeomorphisms of $W$ does not contain any anti-symplectic diffeomorphisms.
\end{lem}

\begin{proof}
Suppose that $W$ is not closed, and $\psi^* \omega = \pm \omega$ for a diffeomorphism $\psi$.
Since the sign on the right-hand side is independent of $x \in W$, and every diffeomorphism is equal to the identity outside a compact subset of the interior of $W$, $\psi$ cannot be anti-symplectic.

Next suppose that $W$ is closed.
Since the group of homeomorphisms of $W$ is locally contractible with respect to the $C^0$-topology, any two homeomorphisms that are sufficiently $C^0$-close induce the same action on the real cohomology groups of $W$.
In particular, the closure of $\Symp (W,\omega)$ in $\Diff (W)$ fixes the non-zero cohomology class $[\omega]$, and thus it does not contain any anti-symplectic diffeomorphisms.

For an alternate proof for closed symplectic manifolds, recall that the $C^0$-closure of $\Symp (W,\omega)$ is contained in the group $\Diff (W,\omega^n)$ of volume preserving diffeomorphisms.
But if $\psi^* \omega = - \omega$, then $\psi^* (\omega^n) = (- 1)^n \omega^n$, and thus if $n$ is odd, the group $\Diff (W,\omega^n)$ does not contain any anti-symplectic diffeomorphisms.
Then one proceeds to first proving Theorem~\ref{thm:symp-rigidity} for $n$ odd, and proves Theorem~\ref{thm:symp-rigidity} for even $n$ afterward.

If $n$ is even, one applies the above argument to a product manifold $(W \times \Sigma, \omega \times \sigma)$, where $\Sigma$ is a closed surface with an area form $\sigma$.
Assume a diffeomorphisms $\psi$ is contained in the $C^0$-closure of $\Symp (W,\omega)$, and let $\psi_k \in \Symp (W,\omega)$ be a sequence of symplectic diffeomorphisms that converges uniformly to $\psi$.
Since $n + 1$ is odd, $\Symp (W \times \Sigma, \omega \times \sigma)$ is $C^0$-closed in $\Diff (W \times \Sigma)$ by what we have already proved, and in particular the $C^0$-limit $\psi \times \id$ of the sequence $\psi_k \times \id$ lies in $\Symp (W \times \Sigma, \omega \times \sigma)$.
But then $\psi$ must preserve $\omega$.
\end{proof}

For $W = \R^{2 n}$, one obtains the common formulation of Gromov's alternative.

\begin{lem}
The group $\Symp (\R^{2 n}, \omega_0)$ of symplectic diffeomorphisms is either $C^0$-closed in the group of all diffeomorphisms of $\R^{2 n}$, or its $C^0$-closure is the full group $\Diff (\R^{2 n}, \omega_0^n)$ of volume preserving diffeomorphisms.
\end{lem}

\begin{proof}
If $\Symp (\R^{2 n},\omega_0)$ is not $C^0$-closed, then by Gromov's maximality theorem and Lemma~\ref{lem:no-anti-symp}, it contains the connected component of the identity in $\Diff (W,\omega^n)$.
But the $C^0$-closure of the latter generates the entire group.
Indeed, let $R_t$ denote the map $x \mapsto t x$, and let $\varphi$ be a compactly supported volume preserving diffeomorphism of $\R^{2 n}$.
Then for $t > 0$ sufficiently small, the time-$t$ map of the volume preserving isotopy $R_t \circ \varphi \circ R_t^{- 1}$ is supported in a small neighborhood of the origin, and thus is $C^0$-close to the identity.
\end{proof}

\begin{proof}[Proof of Theorem~\ref{thm:symp-rigidity}]
Let $U \subset W$ be the domain of a Darboux chart, and identify $U$ with an open subset of $\R^{2 n}$ with its standard symplectic structure.
For $r > 0$ sufficiently small, there exists a closed ball $B^{2 n} (r) \subset U$, and a volume preserving diffeomorphism $\varphi$ of $U$ that is the identity near the boundary, so that the image of $B^{2 n} (r)$ is contained in $Z^{2 n} (r / 4) \cap U$.
This diffeomorphism $\varphi$ can be chosen to be generated by a divergence-free vector field with support in $U$, and in particular, it can be viewed as an element of $\Diff (W,\omega^n)$.

Suppose there exists a sequence $\varphi_k$ of symplectic diffeomorphisms of $(W,\omega)$ that converges uniformly to $\varphi$.
Their restrictions to $B^{2 n} (r)$ are symplectic embeddings into $Z^{2 n} (r / 2)$ for $k$ sufficiently large.
But that contradicts Gromov's non-squeezing theorem.
Then by Gromov's maximality theorem, the $C^0$-closure of $\Symp (W,\omega)$ cannot contain any diffeomorphisms that are neither symplectic nor anti-symplectic.
Finally by Lemma~\ref{lem:no-anti-symp}, $\Symp (W,\omega)$ is $C^0$-closed inside the group $\Diff (W)$.
\end{proof}

\begin{thm}[{Gromov's maximality theorem \cite[Section~3.4.4~H]{gromov:pdr86}}] \label{thm:maximality}
Let $(M,\xi)$ be a connected contact manifold.
Suppose that a (not necessarily closed) subgroup $G \subset \Diff (M)$ contains all contact diffeomorphisms that are isotopic to the identity through an isotopy of contact diffeomorphisms, and that $G$ contains an element $\psi$ that does not preserve the (un-oriented) contact distribution $\xi$, i.e.\ $\psi_* \xi \not= \xi$.
Then $G$ contains the connected component of the identity of the group $\Diff (M)$.
\end{thm}

Again the first step in the proof uses linear algebra to show that there are enough diffeomorphisms $\psi_1, \ldots, \psi_N \in G$ so that every vector field $X$ can be decomposed uniquely into a sum $X = X_1 + \ldots + X_N$, where the vector field $X_k$ is now contact with respect to the contact structure $(\psi_k)_* \xi$.
The linearization of the analogous map $D \colon \Diff (M,\xi) \times \ldots \times \Diff (M,\xi) \to \Diff (M)$, $(\varphi_1, \ldots, \varphi_N) \mapsto \psi_1 \circ \varphi_1 \circ \ldots \circ \psi_N \circ \varphi_N$, is again surjective near the identity, and has a right inverse.
The techniques of Gromov's implicit function theorem imply that the image of $D$ contains an open neighborhood of the identity in the subgroup of all diffeomorphisms generated by (time-dependent) vector fields, and this neighborhood again generates the entire connected component of the identity in the group of all diffeomorphisms of $M$.

\begin{lem} \label{lem:no-anti-contact}
Let $(M,\xi)$ be a connected contact manifold.
The $C^0$-closure of the group $\Diff (M,\xi)$ of contact diffeomorphisms in the group $\Diff (M)$ does not contain any diffeomorphisms that reverse the coorientation of $\xi$.
\end{lem}

\begin{proof}
If $M$ is not closed, the argument is the same as in the symplectic case.

If $M$ is closed, one may use the fact that every element of $\Diff (M,\xi)$ preserves the orientation induced by the volume form $\alpha \wedge (d\alpha)^{n - 1}$ (which is independent of the choice of contact form $\alpha$ with $\ker \alpha = \xi$), and that this property is preserved under $C^0$-limits.
If $n$ is odd, it follows at once that the $C^0$-closure of $\Diff (M,\xi)$ does not contain any diffeomorphisms that reverse the coorientation of $\xi$.
Then one proceeds to proving the main theorem for $n$ odd, and again apply it in the proof for even $n$.

If $n$ is even, one needs in addition a theorem of Bourgeois \cite{bourgeois:odt02} that for a closed contact manifold $(M,\xi)$, the product $M \times T^2$ also admits a contact structure with the property that for each point $p \in T^2$, the submanifold $M \times p \subset M \times T^2$ is contact and moreover contact diffeomorphic to $(M,\xi)$.
Then one concludes the main theorem for even $n$ by the same argument as in the symplectic case.
\end{proof}

\begin{lem}
The group $\Diff (\R^{2 n - 1}, \xi_0)$ of contact diffeomorphisms is either $C^0$-closed or $C^0$-dense in the group of all diffeomorphisms of $\R^{2 n - 1}$.
\end{lem}

\begin{proof}
The diffeomorphism $(x, y, z) \mapsto (t x, t y, t^2 z)$ preserves the standard contact structure $\xi_0$, and thus the proof is similar to the symplectic case.
\end{proof}

In the case of $\R^{2 n} \times S^1$ with its standard contact structure, it is possible to use non-squeezing phenomena \cite{eliashberg:gct06} similar to Theorem~\ref{thm:non-squeezing} to decide Gromov's alternative in favor of rigidity.
However, the domain of a Darboux chart in $\R^{2 n} \times S^1$ does not contain a subset of the form $U \times S^1$ for an open set $U \subset \R^{2 n}$, so that this proof does not generalize to arbitrary contact manifolds.
If $M$ is a three-manifold, then there are non-squeezing phenomena (\cite[Section~3.4.4~F]{gromov:pdr86} or \cite{eliashberg:gct06}) that give rise to an alternate proof of $C^0$-rigidity of contact diffeomorphisms.

\section{The contact shape} \label{sec:contact-shape}
In this section we review the contact shape defined and studied in \cite{eliashberg:nio91}.
We only present the elements of this invariant that are necessary to prove contact rigidity in the next section, and refer the reader to the original paper for a more general discussion.

Let $T^n = \R^n / \Z^n$ be an $n$-dimensional torus, and $T^* T^n = T^n \times \R^n$ its cotangent bundle with the canonical symplectic structure $\omega = d \lambda$, where $\lambda = \sum p_i dq_i$, and where $(q_1, \ldots, q_n)$ and $(p_1, \ldots, p_n)$ denote coordinates on the base $T^n = \R^n / \Z^n$ and the fiber $\R^n$, respectively.
An embedding $f \colon T^n \hookrightarrow U$ into an open subset $U \subset T^* T^n$ is called Lagrangian if $f^* \omega = 0$, and the cohomology class $\lambda_f = [f^* \lambda]$ in $H^1(T^n,\R)$ is called its $\lambda$-period.
Choose a homomorphism $\tau \colon H^1 (U,\R) \to H^1 (T^n,\R)$.
The $\tau$-shape of $U$ is the subset $I (U, \tau)$ of $H^1 (T^n,\R)$ that consists of all points $z \in H^1 (T^n,\R)$ such that there exists a Lagrangian embedding $f \colon T^n \hookrightarrow U$ with $f^* = \tau$ and $z = \lambda_f$.

Let $(M,\xi = \ker \alpha)$ be a connected contact manifold, and denote by $M \times \R_+$ the symplectization of $(M,\alpha)$ with the symplectic structure $\omega = d(t p^* \alpha)$, where $p \colon M \times \R_+ \to M$ is the projection to the first factor, and $t$ is the coordinate on the factor $\R_+ = (0,\infty)$.
(The symplectization can also be defined as the manifold $M \times \R$ with coordinate $\theta$ on the second factor and the change of coordinates $t = e^\theta$.)
Up to symplectic diffeomorphisms, the symplectization depends only on the contact structure $\xi$, and not on the particular choice of contact form $\alpha$.
A contact embedding $\varphi \colon (M_1,\xi_1 = \ker \alpha_1) \to (M_2,\xi_2 = \ker \alpha_2)$ with $\varphi^* \alpha_2 = f \alpha_1$ induces an ($\R_+$-equivariant) symplectic embedding $(M_1 \times \R_+, d(t p^* \alpha_1)) \to (M_2 \times \R_+, d(t p^* \alpha_2))$ given by $(x,t) \mapsto (\varphi (x),t / f (x))$.

Let $ST^* T^n \subset T^* T^n$ denote the unit cotangent bundle of $T^n$ with its canonical contact structure $\xi = \ker \alpha$, where $\alpha$ is the restriction of the canonical $1$-form $\lambda$ on $T^* T^n$.
The symplectization of $ST^* T^n$ is diffeomorphic to $T^* T^n$ minus the zero section with its standard symplectic structure $\omega$.
Denote by $U$ an open subset of $ST^* T^n$.
Since $H^1 (U \times \R_+,\R)$ is canonically isomorphic to $H^1 (U,\R)$, a given homomorphism $\tau \colon H^1 (U,\R) \to H^1 (T^n ,\R)$ can be identified with a homomorphism $H^1 (U \times \R_+,\R) \to H^1 (T^n,\R)$.
Thus the symplectic shape $I (U \times \R_+, \tau)$ of the symplectization is a contact invariant of the contact manifold $(U,\xi)$.
It is easy to see from the definition that $I (U \times \R_+, \tau)$ is a cone without the vertex in $H^1 (T^n,\R)$.
Thus it is convenient to projectivize (in the oriented sense of identifying vectors that differ by a positive scalar factor) the invariant, and define the contact ($\tau$-)shape of $U$ by
	\[ I_C (U, \tau) = PI (U \times \R_+, \tau) = I (U \times \R_+, \tau) / \R_+ \subset H^1 (T^n,\R) / \R_+ = PH^1 (T^n,\R). \]

\begin{pro}[{\cite[Proposition~3.2.2]{eliashberg:nio91}}] \label{pro:contact-inv}
Let $U_1$ and $U_2$ be two open subsets of the unit cotangent bundle $ST^* T^n$ with its canonical contact structure, and $\varphi \colon U_1 \to U_2$ be a contact embedding.
Then $I_C (U_1, \tau) \subset I_C (U_2, \tau \circ \varphi^*)$.
In particular, if $\varphi$ is a contact diffeomorphism, then $I_C (U_1, \tau) = I_C (U_2, \tau \circ \varphi^*)$.
\end{pro}

The above decomposition of $T^* T^n = T^n \times \R^n$ restricts to the decomposition $ST^* T^n = T^n \times S^{n - 1}$.
Choose cohomology classes $[dq_1], \ldots, [dq_n]$ as a basis of $H^1 (T^n,\R)$ to identify it with the fiber $\R^n$ of the fibration $T^* T^n \to T^n$.
This gives rise to an identification of the (oriented) projectivized group $PH^1 (T^n,\R)$ with the fiber $S^{n -1}$ of the fibration $ST^* T^n = T^n \times S^{n - 1} \to T^n$.

\begin{pro}[{\cite[Proposition~3.4.1]{eliashberg:nio91}}] \label{pro:a-shape}
Let $A \subset S^{n - 1}$ be a connected open subset, and $U = T^n \times A \subset ST^* T^n$.
For $a \in A$, denote by $\iota_a \colon T^n = T^n \times a \hookrightarrow ST^* T^n$ the canonical embedding.
Then $I_C (U, \iota_a^*) = A$.
\end{pro}

\section{$C^0$-rigidity of contact diffeomorphisms} \label{sec:rigidity}
In this section we construct a diffeomorphism that is an element of the connected component of the identity in the group of diffeomorphisms of $M$, but on the other hand is not contained in the $C^0$-closure of the group of contact diffeomorphisms.
Together with Gromov's maximality theorem and Lemma~\ref{lem:no-anti-contact}, this yields a proof of the main theorem.

Recall that for closed manifolds, the $C^0$-topology is given by the compact-open topology.
It is also a metric topology induced by uniform convergence with respect to some auxiliary Riemannian metric.
(This is not a complete metric, but the $C^0$-topology is also induced by a complete metric.)
For open manifolds, the $C^0$-topology is defined as a direct limit of the $C^0$-topologies on compact subsets of the interior.

As in the previous section, denote by $ST^* T^n = T^n \times S^{n - 1}$ the unit cotangent bundle of an $n$-dimensional torus with its standard contact structure induced by the contact form $\sum_{i = 1}^n a_i dt_i$, where $(a_1, \ldots, a_n)$ denote coordinates on the sphere, and $(t_1, \ldots, t_n)$ are coordinates on the torus.

\begin{lem}
Let $(M,\xi)$ be a contact manifold, and $a \in S^{n -1}$ be a point.
Then there exists a neighborhood $A$ of $a$ in $S^{n - 1}$ and a contact embedding of $T^n \times A$ with its standard contact structure into $(M,\xi)$.
\end{lem}

\begin{proof}
Since contact diffeomorphisms act transitively on the underlying manifold, it suffices to prove the lemma for a point $a = (a_1, \ldots, a_n) \in S^{n -1}$ with $0 < a_i < 1$ for all $i$.
We will show that a neighborhood of $T^n \times a$ is contact diffeomorphic to an open subset of $S^1 \times \R^{2 n - 2}$ with the contact structure induced by the contact form $dz + \sum_{i = 1}^{n -1} (x_i dy_i - y_i dx_i) = dz + \sum_{i = 1}^{n -1} r_i^2 d\theta_i$, where $z \in S^1$, and $x_i = r_i \cos \theta_i$ and $y_i = r_i \sin \theta_i$ are coordinates on $\R^{2 n - 2}$.
This form is diffeomorphic to the standard contact form $dz - \sum_{i = 1}^{n -1} y_i dx_i$.
We moreover prove that a neighborhood of $S^1 \times 0$ in $S^1 \times \R^{2 n - 2}$ is contact diffeomorphic to an open subset of $(M,\xi)$.
Combining the two maps gives rise to the desired embedding.

By Darboux's theorem, $(M,\xi)$ contains an embedded transverse knot $S^1 \hookrightarrow M$.
By the contact neighborhood theorem, a sufficiently small neighborhood of this knot is contact diffeomorphic to an open neighborhood $S^1 \times U$ of $S^1 \times 0 \subset S^1 \times \R^{2 n - 2}$ with its standard contact structure.
That proves the second claim.

Let $\O$ be the open subset of $S^1 \times \R^{2 n - 2}$ on which $r_i > 0$ for all $i$, and define a function $r \colon \O \to (0,1)$ by $r = (1 + \sum_{i = 1}^{n - 1} r_i^2)^{- \frac{1}{2}}$.
The map $\O \to T^n \times S^{n - 1}$ given by
\begin{align*}
	(z, r_1, \ldots, r_{n - 1}, \theta_1, \ldots, \theta_{n -1}) & \mapsto (t_1, \ldots, t_n, a_1, \ldots, a_n) \\
	& = (z, \theta_1, \ldots, \theta_{n - 1}, r^2, (r \cdot r_1)^2, \ldots, (r \cdot r_{n - 1})^2)
\end{align*}
is a contact diffeomorphism onto the subset of $T^n \times S^{n -1}$ on which all spherical coordinates are positive.
Since $S^1 \times U \cap \O$ is non-empty, the proof is complete.
\end{proof}

\begin{proof}[Proof of Theorem~\ref{thm:contact-rigidity}]
By the lemma, the contact manifold $(M,\xi)$ contains an open subset that is diffeomorphic to $T^n \times A \subset ST^* T^n$ with its standard contact structure.
Let $B (r) \subset A$ be an open ball with center $a \in S^{n - 1}$ and radius $r > 0$ with respect to the standard metric on $S^{n - 1}$, and $\varphi_1$ be the time-one map of a smooth vector field that is compactly supported in $A$, and maps the ball $B (r)$ to the ball $B (r / 4)$.
The diffeomorphism $\id \times \varphi_1$ is compactly supported in $T^n \times A$, and thus extends to a diffeomorphism $\varphi$ of $M$ that is isotopic to the identity.
We will use Eliashberg's shape invariant to show that $\varphi$ can not be approximated uniformly by contact diffeomorphisms.
Then by Gromov's maximality theorem combined with Lemma~\ref{lem:no-anti-contact}, the $C^0$-closure of $\Diff (M,\xi)$ does not contain any diffeomorphisms that do not preserve the (oriented) contact structure $\xi$.

Suppose $\varphi_k$ is a sequence of contact diffeomorphisms that converges uniformly to $\varphi$.
For $k$ sufficiently large, the restriction $\iota_k$ of $\varphi_k$ to $T^n \times B (r)$ is a contact embedding into $T^n \times B (r / 2)$, and thus by Proposition~\ref{pro:contact-inv}, the shape invariant satisfies the relation $I_C (T^n \times B (r), \iota_a^*) \subset I_C (T^n \times B (r / 2), \iota_a^* \circ \iota_k^*)$.
For $k$ sufficiently large, the map $\iota_k$ is homotopic to the restriction of $\varphi$ to $T^n \times B (r)$, and the latter fixes the submanifold $T^n \times a$.
Thus the induced map $\iota_a^* \circ \iota_k^*$ equals $(\iota_k \circ \iota_a)^* = \iota_a^*$.
In particular,
	\[ I_C (T^n \times B (r), \iota_a^*) \subset I_C (T^n \times B (r / 2), \iota_a^* \circ \iota_k^*) = I_C (T^n \times B (r / 2), \iota_a^*) = B (r / 2) \]
by Proposition~\ref{pro:a-shape}.
On the other hand, $I_C (T^n \times B (r), \iota_a^*) = B (r)$, and we arrive at a contradiction.
That means the sequence $\varphi_k$ cannot exist, hence the proof.
\end{proof}

As a final remark, note that Theorem~\ref{thm:contact-rigidity} implies that the set of diffeomorphisms that preserve $\xi$ but reverse its coorientation is also $C^0$-closed in the group $\Diff (M)$ of all diffeomorphisms.
Indeed, if this set is non-empty, then it is equal to $\varphi \cdot \Diff (M,\xi)$ for some diffeomorphism $\varphi$ that reverses the coorientation of $\xi$.

\bibliography{rigidity}
\bibliographystyle{amsalpha}
\end{document}